\documentclass[a4paper,10pt]{article}
\usepackage{amsmath,amssymb,amsfonts,amsthm}
\usepackage{graphicx}
\usepackage{pstricks}
\usepackage[colorlinks,plainpages,citecolor=blue]{hyperref}
\date{\today}
\usepackage{enumitem}
\usepackage{textcomp}
\usepackage{caption}

\newtheorem{theorem}{Theorem}[section]
\newtheorem{lemma}[theorem]{Lemma}

\newtheorem{corollary}[theorem]{Corollary}
\newtheorem{proposition}[theorem]{Proposition}
\newtheorem*{exemple}{Example}

\newcommand{\R}{\mathbb R }
\newcommand{\Z}{\mathbb Z }
\newcommand{\N}{\mathbb N }

\newcommand{\GGBS}{$vGBS{}$}
\newcommand{\GBS}{$GBS{}$}

\usepackage{pdfsync}

\title{Abelian JSJ decomposition of graphs of free abelian groups}
\author{Benjamin Beeker}

\begin{document}

\maketitle
\begin{abstract}
 A group $G$ is a \GGBS~group if it admits a decomposition as a finite graph of groups with all edge and vertex groups finitely generated and free abelian. We construct the JSJ decomposition of a \GGBS~group over abelian groups. We prove that this decomposition is explicitly computable, and may be obtained by local changes on the initial graph of groups.
\end{abstract}

\section{Introduction}
The theory of JSJ decomposition starts with the work of Jaco-Shalen and Johansson on orientable irreducible closed $3$-manifolds giving a canonical family of 2-dimensional tori. Kropholler first introduced the notion into group theory giving a JSJ decomposition for some Poincar\'e duality groups \cite{Krop}. Then Sela gave a construction for torsion-free hyperbolic groups \cite{Sela}. This notion has been more generally developed by Rips and Sela \cite{RiSe}, Dunwoody and Sageev \cite{DunSa2}, Fujiwara and Papasoglu \cite{FuPa} for various classes of groups. In \cite{GL3a}, Guirardel and Levitt generalize the object by introducing the definition of JSJ deformation space, proving the existence of this space for finitely presented groups.

In this paper we consider the following class of groups.

Let $\Gamma$ be a finite graph of groups with all vertex groups finitely generated free abelian. Let $G$ be the fundamental group of $\Gamma$. If the rank of all edge and vertex groups is equal to a fixed integer $n$, we call such a group a $\GBS_n$ group, standing for Generalized Baumslag-Solitar groups of rank $n$. When the rank is variable, we call such a group a \GGBS~group. The goal of this paper is to describe the JSJ decomposition of $G$ over abelian groups and to give a way to construct it. 

A decomposition of a group $G$ is a graph of groups with fundamental group $G$. To define what a JSJ decomposition is, we need the notion of universally elliptic subgroups. Given a group $G$ and a decomposition $\Gamma$ of $G$, a subgroup $H$ is \textit{elliptic} if 
the group $H$ is conjugate to a subgroup of a vertex group of $\Gamma$. Given a class of subgroups $\mathcal A$ of $G$, a subgroup $H\subset G$ is \textit{universally elliptic}, if $H$ is elliptic in every decomposition of $G$ as a graph of groups with edge groups in $\mathcal A$. A decomposition is universally elliptic if all edge groups are universally elliptic.

A decomposition $\Gamma$ \textit{dominates} another decomposition $\Gamma'$, if every elliptic group of $\Gamma$ is elliptic in $\Gamma'$. A decomposition is a JSJ decomposition if it is universally elliptic, and it dominates every other universally elliptic decomposition.
Then given a vertex $v$ in a JSJ decomposition, either $G_v$, the vertex group of $v$, is universally elliptic, we say $v$ is \textit{rigid}, or $G_v$ is not then we say $v$ is \textit{flexible}. 

For instance, looking at the JSJ decomposition of a torsion-free hyperbolic group over cyclic groups described by Sela in \cite{Sela} and Bowditch \cite{bow}, the flexible vertex groups are exactly the surface groups.

The defining decomposition $\Gamma$ of a \GGBS~group $G$ is a good approximation of the JSJ. For example, for a $\GBS_n$ group, the graph of groups $\Gamma$ is a JSJ decomposition whenever the associated Bass-Serre tree $T$ (which is locally finite in this case) is not a line (\cite{for},\cite{GL3a}) or equivalently whenever $G$ is not polycyclic.

In the general case three kinds of local changes may be needed to obtain the JSJ from $\Gamma$.

\begin{itemize}
 \item When the quotient of a vertex group $G_v$ of the decomposition $\Gamma$ by the group $\tilde G_v$ generated by all adjacent edge groups is virtually cyclic, then the vertex is blown into a loop (see figure \ref{figurecasexpansion} for the simplest example). 

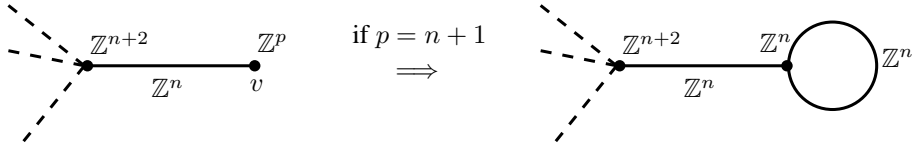
\begin{figure}[!ht]

\begin{center}

\begin{pspicture}(0,-0.92)(12.104218,0.92)
\psline[linewidth=0.04cm,dotsize=0.07055555cm 2.0]{*-*}(1.0395312,0.1)(3.239531,0.1)
\psline[linewidth=0.04cm,dotsize=0.07055555cm 2.0]{*-*}(8.039532,0.1)(10.239531,0.1)
\pscircle[linewidth=0.04,dimen=outer](10.839532,0.1){0.6}
\usefont{T1}{ptm}{m}{n}
\rput(2.09375,-0.195){$\Z^n$}
\usefont{T1}{ptm}{m}{n}
\rput(1.46375,0.405){$\Z^{n+2}$}
\usefont{T1}{ptm}{m}{n}
\rput(3.4737499,0.405){$\Z^{p}$}
\usefont{T1}{ptm}{m}{n}
\rput(3.2637498,-0.155){$v$}
\usefont{T1}{ptm}{m}{n}
\rput(8.46375,0.405){$\Z^{n+2}$}
\usefont{T1}{ptm}{m}{n}
\rput(9.093749,-0.195){$\Z^n$}
\usefont{T1}{ptm}{m}{n}
\rput(11.693749,0.205){$\Z^n$}
\usefont{T1}{ptm}{m}{n}
\rput(10.093749,0.405){$\Z^n$}
\usefont{T1}{ptm}{m}{n}
\rput(5.3737504,0.005){$\Longrightarrow$}
\psline[linewidth=0.04cm,linestyle=dashed,dash=0.16cm 0.16cm](8.0,0.1)(7.2,-0.9)
\psline[linewidth=0.04cm,linestyle=dashed,dash=0.16cm 0.16cm](8.0,0.1)(7.0,0.3)
\psline[linewidth=0.04cm,linestyle=dashed,dash=0.16cm 0.16cm](8.0,0.1)(7.0,0.9)
\psline[linewidth=0.04cm,linestyle=dashed,dash=0.16cm 0.16cm](1.0,0.1)(0.2,-0.9)
\psline[linewidth=0.04cm,linestyle=dashed,dash=0.16cm 0.16cm](1.0,0.1)(0.0,0.3)
\psline[linewidth=0.04cm,linestyle=dashed,dash=0.16cm 0.16cm](1.0,0.1)(0.0,0.9)
\usefont{T1}{ptm}{m}{n}
\rput(5.405625,0.485){if $p=n+1$}
\end{pspicture}

\end{center}

\caption{Expansion of the vertex $v$}
\label{figurecasexpansion}
\end{figure}

To be more specific, in figure \ref{figurecasexpansion}, if $p=n$, then the edge group has finite index in the vertex group of $v$ and $v$ is rigid. If $p=n+1$, there exists a unique non-trivial splitting of $v$ which leaves the $\Z^n$ part elliptic. The vertex must be blown into a loop as in the figure, and the new vertex is then rigid. If $p\geq n+2$ then $v$ is flexible.

\item Conversely some loops must be collapsed. 

A loop $l$ based at a vertex $v$ is called a $1-1$ loop if both inclusion maps of its edge group into $G_v$ are bijections. The fundamental group of the subgraph of groups composed of $v$ and $l$ is a semi-direct product $\Z^n\rtimes_\varphi \Z$. Some of the $1-1$ loops are collapsed, depending on $\varphi$ and on the other edges adjacent to $v$. 

For example, in figure \ref{figurecas1-1}, the loop must be collapsed if and only if $k<n$ : if $k<n$, there are many decompositions of $\Z^{n+1}=\Z^n\rtimes_{id} \Z$ as an HNN extension leaving $\Z^k$ in the vertex group. If $k=n$, there is exactly one, and its edge group is universally elliptic.

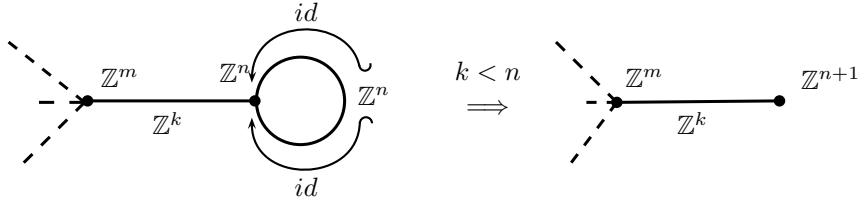
\begin{figure}[!ht]
\begin{center}
\begin{pspicture}(0,-1.5229688)(11.734219,1.5229688)
\psline[linewidth=0.04cm,dotsize=0.07055555cm 2.0]{*-*}(8.0,0.12453125)(10.139532,0.1440624)
\psline[linewidth=0.04cm,dotsize=0.07055555cm 2.0]{*-*}(1.0395312,0.1440624)(3.239531,0.1440624)
\pscircle[linewidth=0.04,dimen=outer](3.8395312,0.1440624){0.6}
\usefont{T1}{ptm}{m}{n}
\rput(8.993751,-0.15093762){$\Z^k$}
\usefont{T1}{ptm}{m}{n}
\rput(8.36375,0.44906238){$\Z^{m}$}
\usefont{T1}{ptm}{m}{n}
\rput(10.81375,0.44906238){$\Z^{n+1}$}
\usefont{T1}{ptm}{m}{n}
\rput(1.46375,0.44906238){$\Z^{m}$}
\usefont{T1}{ptm}{m}{n}
\rput(2.09375,-0.15093762){$\Z^k$}
\usefont{T1}{ptm}{m}{n}
\rput(4.80375,0.2090624){$\Z^{n}$}
\usefont{T1}{ptm}{m}{n}
\rput(2.9837499,0.51109374){$\Z^{n}$}
\psline[linewidth=0.04cm,linestyle=dashed,dash=0.16cm 0.16cm](1.0,0.12453125)(0.2,-0.67546874)
\psline[linewidth=0.04cm,linestyle=dashed,dash=0.16cm 0.16cm](1.0,0.12453125)(0.4,0.12453125)
\psline[linewidth=0.04cm,linestyle=dashed,dash=0.16cm 0.16cm](1.0,0.12453125)(0.0,0.9245312)
\psline[linewidth=0.04cm,linestyle=dashed,dash=0.16cm 0.16cm](8.0,0.12453125)(7.4,-0.67546874)
\psline[linewidth=0.04cm,linestyle=dashed,dash=0.16cm 0.16cm](8.0,0.12453125)(7.6,0.12453125)
\psline[linewidth=0.04cm,linestyle=dashed,dash=0.16cm 0.16cm](8.0,0.12453125)(7.2,0.9245312)
\usefont{T1}{ptm}{m}{n}
\rput(6.3028126,0.52953125){$k<n$}
\rput(6.3028126,0.02953125){$\Longrightarrow$}
\psbezier[linewidth=0.027999999,arrowsize=0.05291667cm 2.0,arrowlength=1.4,arrowinset=0.4]{->}(4.62,-0.13546875)(4.6030965,-0.97302973)(3.2,-0.97546875)(3.2,-0.07790777)
\psbezier[linewidth=0.027999999,arrowsize=0.05291667cm 2.0,arrowlength=1.4,arrowinset=0.4]{->}(4.6141744,0.6045312)(4.4678555,1.2845312)(3.2322712,1.2645313)(3.1986957,0.38453126)
\psarc[linewidth=0.027999999](4.7002497,-0.15262192){0.07975015}{0.0}{180.0}
\psarc[linewidth=0.027999999](4.687334,0.61769086){0.07315958}{179.04517}{19.885164}
\usefont{T1}{ptm}{m}{n}
\rput(3.9114063,-1.0004688){$id$}
\usefont{T1}{ptm}{m}{n}
\rput(3.9114063,1.3295312){$id$}
\end{pspicture} 

\end{center}

\caption{Collapse of a $1-1$ loop}
\label{figurecas1-1}
\end{figure}

\item A similar phenomenon occurs for edges which are not loops and whose group has index $2$ in each adjacent vertex group (see figure \ref{figurecas2-2}). We call an edge of this type a $2-2$ edge.

\begin{figure}[!ht]
\begin{center}
\scalebox{1} 
{
\begin{pspicture}(1,-1.548125)(12.882812,1.548125)
\usefont{T1}{ptm}{m}{n}
\rput(8.301406,0.2739062){$\Longrightarrow$}
\usefont{T1}{ptm}{m}{n}
\rput(5.965938,0.8646874){$\Z^2=\langle c,d\rangle$}
\psline[linewidth=0.04cm](5.2367187,0.32968745)(7.036719,0.7296874)
\psline[linewidth=0.04cm](5.2367187,0.32968745)(6.4367185,-1.0703127)
\psline[linewidth=0.04cm](5.2367187,0.32968745)(6.836718,-0.27031255)
\psline[linewidth=0.04cm](2.76,0.348125)(5.2367187,0.32968745)
\psdots[dotsize=0.12](5.2567186,0.32968745)
\psline[linewidth=0.04cm](10.4167185,0.18968745)(12.216719,0.58968747)
\psline[linewidth=0.04cm](10.4167185,0.18968745)(11.616718,-1.2103126)
\psline[linewidth=0.04cm](10.4167185,0.18968745)(12.016719,-0.41031256)
\psline[linewidth=0.04cm](10.4167185,0.18968745)(9.216719,1.1896875)
\psline[linewidth=0.04cm](10.4167185,0.18968745)(8.616718,-0.6103125)
\psdots[dotsize=0.12](10.4167185,0.18968745)
\usefont{T1}{ptm}{m}{n}
\rput(3.3459377,0.8846874){$\Z^2=\langle a,b\rangle$}
\psline[linewidth=0.04cm](2.8367188,0.32968745)(1.6367185,1.3296874)
\psline[linewidth=0.04cm](2.8367188,0.32968745)(1.0357811,-0.47031257)
\psdots[dotsize=0.12](2.8367188,0.32968745)
\usefont{T1}{ptm}{m}{n}
\rput(4.2523437,-0.06531255){$a^2=c^2$}
\usefont{T1}{ptm}{m}{n}
\rput(10.7223425,0.73468745){$\langle a,c\rangle\times\langle b\rangle$}
\usefont{T1}{ptm}{m}{n}
\rput(1.3923438,-0.6853125){$\langle b^4\rangle$}
\usefont{T1}{ptm}{m}{n}
\rput(1.4323436,0.8546875){$\langle b^3\rangle$}
\usefont{T1}{ptm}{m}{n}
\rput(9.532343,1.3546875){$\langle b^3\rangle$}
\usefont{T1}{ptm}{m}{n}
\rput(8.992344,-0.9053124){$\langle b^4\rangle$}
\usefont{T1}{ptm}{m}{n}
\rput(6.772344,-1.2453125){$\langle d^2\rangle$}
\usefont{T1}{ptm}{m}{n}
\rput(7.132344,-0.4253125){$\langle d^7\rangle$}
\usefont{T1}{ptm}{m}{n}
\rput(6.732344,0.3346875){$\langle d^3\rangle$}
\usefont{T1}{ptm}{m}{n}
\rput(12.032344,-1.3253125){$\langle d^2\rangle$}
\usefont{T1}{ptm}{m}{n}
\rput(12.512343,-0.5453125){$\langle d^7\rangle$}
\usefont{T1}{ptm}{m}{n}
\rput(12.372344,0.2146875){$\langle d^3\rangle$}
\usefont{T1}{ptm}{m}{n}
\rput(4.282344,-0.46531254){$b=d$}
\end{pspicture} 
}

\end{center}
\caption{The group $\langle a,c\rangle$ is a Klein bottle group with trivial intersection with the adjacent edge groups.}
\label{figurecas2-2}
\end{figure}
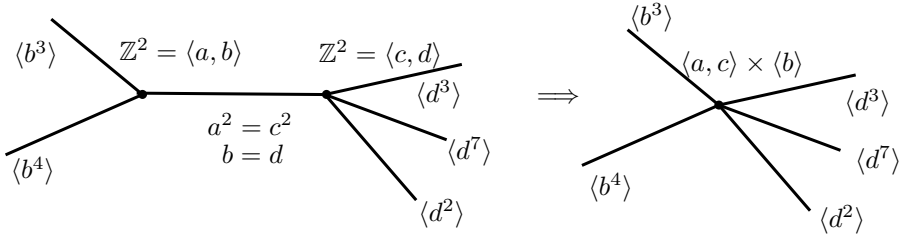
\end{itemize}

Given an edge $e$ of $\Gamma$, call $\Gamma_e$ the subgraph consisting of the single edge $e$ (and its vertices) and $\Pi_e$ its fundamental group.

To obtain the JSJ decomposition of a \GGBS~group over abelian groups, we show that it suffices to expand  vertex groups $G_{v }$ such that $G_{v } / \tilde G_{v }$ is virtually cyclic as in figure \ref{figurecasexpansion}, and collapse edges which are not universally elliptic. These edges have polycyclic groups, and so are $1-1$ loops and $2-2$ edges as in figures \ref{figurecas1-1} and \ref{figurecas2-2}.


\begin{theorem}\label{theoremintro}
    Let $G=\pi_1(\Lambda)$ be a \GGBS~group. For ${v }$ a vertex, let $\tilde G_{v }$ be the subgroup of $G_v$ generated by groups of edges adjacent to ${v }$. A JSJ decomposition of $G$ over abelian groups can be obtained by 
\begin{itemize}
\item  expanding the groups $G_{v }$ such that $G_{v } / \tilde G_{v }$ is virtually cyclic,
\item collapsing $2-2$ edges and $1-1$ loops $e$ which are not universally elliptic.
\end{itemize}
\end{theorem}

As for hyperbolic groups, in a JSJ of a \GGBS~group, the rigid vertices are also easily identifiable. A vertex is rigid if and only if its vertex group is abelian and virtually generated by the adjacent edge groups.

In order to make theorem \ref{theoremintro} more explicit, we shall now describe the edges which are not universally elliptic.

We first describe the edges $e$ for which $\Pi_e$ has a trivial JSJ decomposition. This is the same as giving the JSJ decomposition of polycyclic $\GBS_n$~groups. As (reduced) decompositions of polycyclic groups must have a line (or a point) as Bass-Serre tree, these decompositions have at most one edge, which may be of two types. 

The first are the $2-2$ edges, that is, non-loop edges whose group is of index $2$ in the groups of the two adjacent vertices. Then $\Pi_e$ is isomorphic to a direct product of a $\Z^n$ by either the Klein bottle group $\langle a, b | a^2=b^2 \rangle$, or the twisted Klein bottle group $\langle a, b | a^2b=ba^2, ab^2=b^2a\rangle$.

The second are $1-1$ loops (whose inclusion maps into the vertex group are bijections), and their groups can be seen as semi-direct product $\Z^n \rtimes_\varphi\Z$:
\begin{theorem}\label{1-1intro}
\begin{enumerate}
 \item If $e$ is an edge of type $2-2$ then $\Pi_e$ has a trivial JSJ decomposition over abelian groups.

\item Let us take $\varphi$ in $\text{Aut}(\Z^n)\simeq GL_n(\Z)$.

The group $G_\varphi=\Z^n\rtimes_{\varphi}\Z$ has a trivial JSJ  decomposition if and only if $\varphi$ can be written in a well-chosen basis $(x,h_1,\dots, h_{n-1})$ in one of the following ways:
\begin{enumerate}
    \item $\left(
          \begin{array}{cc}
            1 & 0 \\ 
            p & M
          \end{array}
          \right)$
          with $M$ an $(n-1)\times(n-1)$ matrix of finite order and $p$ in $\Z^{n-1}$,
    \item $\left(
          \begin{array}{cc}
           -1 & 0 \\ 
           p  & Id_{n-1}
          \end{array}
          \right)$ with $p$ in $\Z^{n-1}$.
 \end{enumerate}
In every other case the semi-direct product $\Z^n\rtimes_{\varphi}\Z$ is a JSJ decomposition.
\end{enumerate}
\end{theorem}
This theorem is proved in  section \ref{section22} and \ref{section11}.

By \cite[lemma 4.10]{GL3a}, we know that the edges $e$ which are not universally elliptic are those for which $\Pi_e$ has a trivial JSJ decomposition relative to adjacent edge groups. We therefore have to understand how $\Pi_e$ is embedded in the whole group. More precisely, we look at the way the adjacent edge groups inject into $\Pi_e$. We obtain the following theorem which makes theorem \ref{theoremintro} explicit.

Define a \textit{hyperplane} of $\Z^n$ as the kernel of a morphism from $\Z^n$ to $\Z$.
\begin{theorem}\label{simplesintro}

Let $G=\pi_1(\Gamma)$ be a \GGBS~group. For ${v }$ a vertex, let $\tilde G_{v }$ be the subgroup of $G_v$ generated by groups of edges adjacent to ${v }$. A JSJ decomposition of $G$ over abelian groups can be obtained by 
\begin{itemize}
\item  expanding the groups $G_{v }$ such that $G_{v } / \tilde G_{v }$ is virtually cyclic,
\item collapsing $1-1$ edges $e$ such that 
 $\Pi_e$ may be decomposed as in theorem \ref{1-1intro} with all adjacent edge groups included in the hyperplane $\langle h_1, \dots, h_{n-1}\rangle$. 
\item collapsing $2-2$ edges $e$  with vertices $v$ and $v'$, 
such that there is a hyperplane $H$ of $G_{e }$ which is also a hyperplane of $G_{v }$ and $G_{{v' }}$ and which contains all groups of adjacent edges, $G_{e }$ excepted.
\end{itemize}
\end{theorem}
We prove this proposition in sections \ref{sectioninteraction} and \ref{finalsection}.
We prove in section \ref{algo} that the construction of the JSJ is algorithmic.

From theorem \ref{simplesintro}, we obtain the JSJ decomposition of \GGBS~groups over abelian groups with bounded rank:

\begin{theorem}
 Let $G=\pi_1(\Gamma)$ a \GGBS~group, and $n\in \N$. Suppose that $\Gamma$ is a JSJ decomposition over free abelian groups. Then a JSJ decomposition over free abelian groups of rank $\leq n$ may be obtained by collapsing every edge of $\Gamma$ with group of rank $> n$.
 \end{theorem}


\section{Preliminaries}
Let $G$ be a finitely generated group. We denote by $T$ the Bass-Serre tree associated to a finite graph of groups decomposition $\Gamma$ of $G$. For a vertex $v$ (resp. an edge $e$) of $\Gamma$, we denote $G_v(\Gamma)$ (resp $G_e(\Gamma)$) its stabilizer. Most of the time $\Gamma$ will be omitted.

For $e$ an (unoriented) edge of $T$ with vertices $v$ and $v'$, we call the \textit{type} of $e$ the couple $m-n$ of potentially infinite numbers with $m\leq n$ such that $G_e$ is of index $m$ in $G_v$ and $n$ in $G_{v'}$. For example, an edge whose stabilizer equals the stabilizer of its vertices, is of type $1-1$. This type only depends on the orbit of the edge, so we define the type of an edge of the graph of groups as the type of an edge in $T$ representing it. 
We say that an edge $e$ of $T$ is a loop if its vertices are in the same orbit under the action of $G$. A graph of groups is \textit{reduced} if all edges of type $1-n$ are loops in the graph. Every decomposition may be reduced by contracting successively the edges of type $1-n$ which are not loops.

Given a Bass-Serre tree $T$, an element $g\in G$ is \textit{elliptic} if it fixes a vertex. Otherwise $g$ is said to be \textit{hyperbolic}. The \textit{characteristic space} of $g$ is the minimal subtree of $T$ containing the vertices $v$ such that the distance between $v$ and $g\cdot v$ is minimal (seeing $T$ as a metric space with all edges of lenght $1$).  When $g$ is elliptic, it is the set of all fixed edges and vertices. When $g$ is hyperbolic this is the only line on which $g$ acts by a translation. In this case we call it the \textit{axis} of $g$.
A subgroup $H$ of $G$ is \textit{elliptic} if it fixes a vertex or equivalently when $H$ is finitely generated if all of its elements are elliptic. In the case of finitely generated abelian groups, the ellipticity of a generating set implies the ellipticity of the whole group.

From now on and for the rest of the paper, the decompositions we consider will be over free abelian groups, meaning that every edge stabilizer is finitely generated free abelian.

An element or a subgroup of $G$ is \emph{universally elliptic} if it is elliptic in the Bass-Serre trees of all graph of groups decompositions. An edge is universally elliptic if it carry a universally elliptic group. A decomposition is \textit{universally elliptic} if all edge are universally elliptic. A graph of group decomposition $\Gamma$ dominates an other decomposition $\Gamma'$ if every elliptic subgroups of $\Gamma$ is elliptic in $\Gamma'$. A \textit{JSJ decomposition} is a universally elliptic decomposition which dominates every other universally elliptic decomposition (\cite{GL3a}).

If $e$ is an edge of a Bass Serre tree, with vertices $v_1$ and $v_2$, of type $1-1$ representing a loop $l$ in the graph of groups, let $t$ be a hyperbolic element such that $t\cdot v_1 = v_2$. We call \emph{modulus} of $l$ the linear map $\varphi\in Aut(G_{v_1})$ such that for all $x$ in $G_{v_1}$ we have $\varphi(x)=txt^{-1}$.  As $G_v\simeq G_l \simeq Z^n$, we can see $\varphi$ as an element of $GL_n(\Z)$. Up to conjugacy, the modulus does not depend neither on the choice of $t$ nor on the choice of $e$ representing $l$ but is switched to its inverse if we change the orientation of $l$.

\section{Universally elliptic edges}\label{sectionuniverselle}

 \begin{proposition}\label{centralisateur}

  Let ${G }=\pi_1({\Gamma })$ be a graph of groups and $T$ its Bass-Serre tree. Let ${x }$ be a hyperbolic element of ${G }$. Then the centralizer $C_{G }({x })$ of ${x }$ in ${G }$ is a semi-direct product ${E}\rtimes {H }$ with ${E}$ a subgroup of an edge stabilizer of $T$ and ${H }$ a cyclic subgroup of ${G }$ generated by a hyperbolic element.

  When all edge groups of $\Gamma$ are finitely generated free abelian, the centralizer of a hyperbolic element is a polycyclic group $\Z^n\rtimes\Z$.
 \end{proposition}

 \begin{proof} 

  The group $C_{G }({x })$ acts on the axis of ${x }$ by translation. This action defines a morphism from $C_G({x })$ to $\Z$. The kernel of this morphism fixes the axis pointwise and so belongs to the stabilizer of the axis.
 \end{proof}

  The centralizers of hyperbolic elements have a very specific structure, which is not the case for elliptic ones. This forces most of edge groups of \GGBS~group to be universally elliptic:

 \begin{corollary}\label{rigidite1}

  Let ${G }=\pi_1({\Gamma })$ be a \GGBS~group with $\Gamma$ reduced. Let ${e }$ be an edge of $\Gamma$.
    \begin{enumerate}
     \item If ${e }$ is not a loop of $\Gamma$ and is not a $2-2$ edge then $G_{e }$ is universally elliptic.
     \item If ${e }$ is a loop of $\Gamma$ but is not $1-1$ then $G_{e }$ is universally elliptic.
    \end{enumerate}
 \end{corollary}

\begin{proof}
 Call $\tilde e$ a representative of $e$ in the Bass-Serre tree.
  An abelian group generated by finitely many elliptic elements is elliptic. We just have to show that each element of the edge group $G_{\tilde e}$ is universally elliptic.

  Let ${\tilde v }$ and ${\tilde v' }$ be the endpoints of  ${\tilde e }$. Then $G_{{\tilde v }}\ast_{G_{\tilde e }}G_{{\tilde v' }}$ is contained in the centralizer of $G_{\tilde e }$. If $G_{\tilde e }$ is not of index $\leq 2$ in both ${\tilde v }$ and ${\tilde v' }$, the amalgam contains a free group. So the centralizer of $G_{\tilde e }$ cannot be a polycyclic group. By proposition \ref{centralisateur}, every element of the edge stabilizer $G_{\tilde e }$ is elliptic and $G_{e }$ is universally elliptic.

  There remain two cases: when $e$ is a $2-2$ loop, and when $e$ is a $1-k$ loop with $k>1$. In both cases, let $t\in{G }$ be such that $t\cdot {\tilde v } =  {\tilde v' }$.
   
  If ${e }$ is a $2-2$ loop, let ${\alpha }$ be a square in $G_{\tilde e }$. Then the centralizer of ${\alpha }$ contains $G_{{\tilde v }}$, $tG_{{\tilde v }}t^{-1}$ and $t^{-1}G_{{\tilde v }}t$ and so the group $tG_{{\tilde v }}t^{-1}\ast_{G_{{\tilde e }}}(G_{{\tilde v }}\ast_{t^{-1}G_{{\tilde e } }t} t^{-1} G_{{\tilde v }} t)$. But this amalgamated product is of $2-\infty$ type, the same argument as before works. And so ${\alpha }$ is universally elliptic. Every element of $G_{\tilde e }$ has a universally elliptic square, and $G_{e }$ is universally elliptic.

  If ${e }$ is a $1-k$ loop with $k>1$, then for all $n>0$, the group $G_{\tilde e }$ is a subgroup of $t^nG_{{\tilde v }}t^{-n}$ which is abelian. So its centralizer contains $\bigcup_{n>0} t^nG_{{\tilde v }}t^{-n}$, a strictly increasing union of groups. The centralizer contains a non-finitely generated group and cannot be polycyclic, contradicting proposition \ref{centralisateur}. So $G_{e }$ is universally elliptic.
 \end{proof}

  Graphs of groups consisting of an HNN-extension of type $1-1$ or an amalgamation of type $2-2$ are exactly the reduced ones whose Bass-Serre trees are lines. They are the polycyclic \GGBS~groups. The goal of the next two paragraphs is to study the different decompositions of these groups in order to determine their JSJ.

\section{\texorpdfstring{Edges of type $2-2$}{}} \label{section22}
  In this part, one shows that, for a given n, there are exactly two \GGBS~groups whose graphs of groups have two vertices with vertex group $\Z^n$ linked by an edge of type $2-2$, and that these two groups can be seen as semi-direct products of $\Z^n$ by $\Z$. So these groups have both an action by translation and a dihedral action on $\R$. We shall prove in next section that they both have a trivial JSJ decomposition.

  The first one is the direct product of the Klein bottle group $K=\langle b_0,b_1|\ b_0^2=b_1^2\rangle$ by $E=\Z^{n-1}$. One will call $K$ the \emph{untwisted Klein bottle group}.

  The second one is a twisted version of the first, it can be described as the product of $F=\Z^{n-2}$ and the group $K'$ with presentation $$\langle b_0,b_1| \ b_0^2b_1=b_1b_0^2,\ b_1^2b_0=b_0b_1^2 \rangle.$$ One will call $K'$ the \emph{twisted Klein bottle group}. 

Graph of groups decompositions of $K$ and $K'$ are as in figure \ref{figure2-2}. 

\begin{figure}[!ht]
\begin{center}
 \scalebox{1} 
 {
   \begin{pspicture}(0,-0.34921876)(10.8628125,0.36921874)
     \psline[linewidth=0.04cm,dotsize=0.07055555cm 2.0]{*-*}(1.9809375,-0.32921875)(3.7809374,-0.32921875)
     \usefont{T1}{ptm}{m}{n}
     \rput(1.9423437,0.17578125){$\langle b_0\rangle$}
     \usefont{T1}{ptm}{m}{n}
     \rput(3.7923439,0.17578125){$\langle b_1\rangle$}

     \psline[linewidth=0.04cm,dotsize=0.07055555cm 2.0]{*-*}(6.9809375,-0.32921875)(8.7809374,-0.32921875)
     \usefont{T1}{ptm}{m}{n}
     \rput(6.9423437,0.17578125){$\langle b_0,~b_1^2\rangle$}
     \usefont{T1}{ptm}{m}{n}
     \rput(8.7923439,0.17578125){$\langle b_0^2,~b_1\rangle$}
   \end{pspicture} 
 }
\end{center}
\caption{}
\label{figure2-2}
\end{figure}
  The two sets of groups $K\times E$ and $K'\times F$ will be called \emph{extended (untwisted or twisted) Klein bottle groups}. Like $K$ and $K'$, they have a decomposition in a amalgam of type $2-2$.

\begin{proposition}\label{TFAE}
Let $G$ be a \GGBS~group and $n\in\N$. The following are equivalent:
\begin{enumerate}[label={\roman*})]
 \item \label{un}The group $G$ is an extended Klein bottle group $K\times\Z^{n-1}$ or $K'\times\Z^{n-2}$.
 \item \label{deux}The group $G$ is a semi-direct product $Z^n\rtimes_\varphi Z$ with $\varphi=\left(
       \begin{array}{cc}
         -1 & 0 \\ 
         p  & Id_{n-1}
       \end{array}
      \right)$ in a suitable basis of $\Z^n$.
 \item \label{trois}The group $G$ admits a graph of group decomposition with two vertices carrying groups $\Z^n$ and an edge of type $2-2$.
\end{enumerate}

\end{proposition}

We shall prove successively \ref{un} $\Leftrightarrow$ \ref{deux} and \ref{trois} $\Leftrightarrow$ \ref{un}

\begin{lemma}\label{pre2-2}
 The group $K$ can be seen as a semi-direct product $\Z\rtimes_{-Id} \Z$, and $K'$ as $\Z^2\rtimes_\varphi \Z$ with $\varphi=\left(
       \begin{array}{cc}
         -1 & 0 \\ 
         1  & 1
       \end{array}
      \right)$.
\end{lemma}

\begin{proof}
The proof consists in giving for $K$ and $K'$ a change of presentation, such that the second one is a semi-direct product. 

  In the case of $K$, the change of presentation is well known:
  $$\begin{array} {crcl}
       \psi: 	&\langle b_0,b_1|b_0^2=b_1^2\rangle  	&\stackrel{\sim}{\longrightarrow} 	&\langle t,a_0|tat^{-1}=a^{-1}\rangle  \\
  	        &b_0 					&\longmapsto 				&t\\
 		&b_0b_1^{-1}				&\longmapsto 				&a
     \end{array}.$$

In the second case, the presentation  $$\mathcal P = \langle b_0,b_1|b_0^2b_1=b_1b_0^2,\ b_1^2b_0=b_0b_1^2\rangle$$ of $K'$ can be changed to   $$\mathcal P' = \mathcal \langle t,a_1,a_2|ta_2 t^{-1}=a_2, ta_1 t^{-1} =a_1^{-1}a_2,[a_1,a_2]=1\rangle $$ via the map $$\begin{array} {crcl}
 	\psi:&\mathcal P  	&\stackrel{\sim}{\longrightarrow} 	&\mathcal P'\\
 	      &b_0		&\longmapsto				&t\\
 	      &b_0b_1^{-1} 	&\longmapsto				&a_1\\
 	      &b_0^2b_1^{-2}	&\longmapsto				&a_2
 	      \end{array}.$$
\end{proof}

\begin{corollary}\label{2-2}
  Each group $K\times E$ and $K'\times F$ can be identified with a semi-direct product $\mathbb Z^n \rtimes_\varphi \mathbb Z$ where $\varphi\in\text{Aut}(Z^n)$ has matrix $
      \left(
       \begin{array}{cc}
         -1 & 0 \\ 
         p  & Id_{n-1}
       \end{array}
      \right)$ in a well-chosen basis $(x,h_1,\dots h_{n-1})$ .

Moreover, with this identification: 
\begin{itemize}
            \item $x=b_0b_1^{-1}$,
	    \item for all $h$ in $\langle h_1,\dots h_{n-1}\rangle$, the element $xh$ is hyperbolic in the decomposition in amalgam of type $2-2$,
	    \item $\langle h_1,\dots h_{n-1}\rangle = E$ in the case of $K\times E$,
	    \item $\langle h_1,\dots h_{n-1}\rangle=\langle F,b_0^2b_1^{-2}\rangle$ in the case of $K'\times F$.
           \end{itemize}

  In particular, in the semi-direct product decomposition, seen as an HNN extension, the groups $E$, $F$ and the element $b_0^2b_1^{-2}$ stay elliptic, the element $b_0b_1^{-1}$ is elliptic too. On the opposite $b_0$ and $b_1$ are hyperbolic.
\end{corollary}

One can notice that, in the case of untwisted Klein bottle groups, the element $b_0^2b_1^{-2}$ is trivial.

\begin{proof}
 The changes of presentation described in the previous proof can be extended to $K\times E$ or $K'\times F$ by noticing that $(\Z^p\rtimes_\varphi \Z)\times \Z^q \simeq (\Z^p\times \Z^q) \rtimes_{(\varphi\times Id_q)} \Z$. The four points are easy to check.
\end{proof}

This proves \ref{un} $\Rightarrow$ \ref{deux} in proposition \ref{TFAE}.\\
 The automorphism $\varphi=
      \left(
       \begin{array}{cc}
         -1 & 0 \\ 
         p  & Id_{n-1}
       \end{array}
      \right)$ is conjugate to $
      \left(
       \begin{array}{cc}
         -1 & 0 \\ 
         0  & Id_{n-1}
       \end{array}
      \right)$ if all coordinates of $p$ are even, to $\left(
      \begin{array}{rcc}
         -1 & 0 &0 \\ 
         1  & 1 &0\\
         0  & 0 & Id_{n-2}
      \end{array}
      \right)$ if one is odd. So the semi-direct product $\Z^n\rtimes_\varphi\Z$ is isomorphic to one of $K\times E$ or $K'\times F$.
This implies the \ref{deux} $\Rightarrow$ \ref{un} part.

\ref{un} $\Rightarrow$ \ref{trois} has already been done.

Let us prove \ref{trois} $\Rightarrow$ \ref{un}.

  Let $A$ and $B$ be two copies of $\Z^n$ generated by the sets $\{a_1, \dots, a_n\}$ and $\{b_1,\dots, b_n\}$ respectively. By assumption, ${G }$ has a presentation $$\langle A, B |\ a_1^2=\varphi(a_1^2),\ a_j=\varphi(a_j), 2 \leq j \leq n\rangle $$ where $\varphi$ is an isomorphism between $\langle a_1^2,a_2,\dots, a_n\rangle $ and $\langle b_1^2,b_2,\dots, b_n\rangle $.

  If $\varphi^{-1}(b_1^2)$ is a square in $A$, let $c$ be its square root. Let us take the family
  $\left\{ c,~\varphi^{-1}(b_i), i>1\right\}$ as a basis of $A$. The presentation of ${G }$ is then changed to $$\langle b_1,b_2,\dots, b_n,c \ |\ b_ib_j=b_jb_i,\ c^2=b_1^2,\ cb_j=b_jc\ for\ j\neq 1,\rangle$$ which is a presentation of an extended Klein bottle group.

  If $\varphi^{-1}(b_1^2)$ is not a square in $A$, then $\varphi^{-1}(b_1^2)$ may be written $\prod_{i=1}^na_i^{\nu_i}$, and there exists $r>1$ for which $\nu_r$ is odd. We may assume $r=2$. The element $\tilde a_2 = a_2^{1-\nu_2}\varphi^{-1}(b_1^2)$ is such that its image by $\varphi$ is a square $\tilde b_2^2$ in $B$ with $\tilde b_2$ out of the image of $\varphi$. As $a_2$ can be written $a_2=\tilde{a_2}\cdot \prod_{i=1, i\neq 2}^na_i^{-\nu_i}$, the family $\left\{ \tilde a_2,~a_i,~i\neq 2\right\}$ is a basis of $A$. Taking as generators for $B$ the elements $c_2=\tilde b_2$, $c_1=\varphi(a_1^2) $ and $c_i=\varphi(a_i)$ for $i>2$, the group ${G }$ admits as presentation $$\langle a_i,c_i|a_ia_j=a_ja_i, c_ic_j=c_jc_i, a^2_1=c_1, a_2=c_2^2, a_i=c_i, i>2\rangle.$$ And ${G }$ is isomorphic to $ K'\times\Z^{n-2}$.

\section{\texorpdfstring{Loops of type $1-1$}{}}\label{section11}

The groups we consider here are semi-direct products of $\Z^n$ by $\Z$. The goal of the section is to determine their JSJ decompositions.

For $\varphi$ an element of $Gl_n(\Z)$ we write $G_\varphi$ the group $\Z^n\rtimes_\varphi \Z$. We will call \emph{hyperplane} of $\Z^n$ the kernel of a linear map  $f:\Z^n\twoheadrightarrow \Z$.

 \begin{proposition}\label{simples}
If $\varphi$ cannot be written (up to conjugation) in one of the following ways:
   \begin{enumerate}
    \item $\left(
          \begin{array}{cc}
            1 & 0 \\ 
            p & M
          \end{array}
          \right)$
          with $M$ an $(n-1)\times(n-1)$ matrix of finite order and $p$ in $\Z^{n-1}$,
    \item $\left(
          \begin{array}{cc}
           -1 & 0 \\ 
           p  & Id_{n-1}
          \end{array}
          \right)$ with $p$ in $\Z^{n-1}$.
   \end{enumerate} then the group $G_\varphi=\Z^n\rtimes_\varphi \Z$ has a unique non-trivial (reduced) graph of groups decomposition.

Its JSJ decomposition is the HNN-extension $\Z^n\ast_\varphi$. 
 \end{proposition}

 \begin{lemma}\label{1-1}
Let ${x }$ be an element of the $\Z^n$ part of $G_\varphi$. Assume that there exists a graph of groups decomposition $\Lambda$ of $G_{\varphi }$ for which ${x }$ is hyperbolic. Then there exists a hyperplane ${H }\subset \Z^n$ elliptic in the decomposition $\Lambda$, stable under the action of ${\varphi }$, and such that ${\varphi }_{|{H }}$ has finite order. Moreover, the set of elliptic elements of  ${\Z^n}$ in $\Lambda$ is exactly ${H }$, and if ${\varphi }({x }){x }^{-1}$ does not belong to ${H }$ then ${\varphi }_{|{H }} = Id$. In particular $\varphi$ can be written in one of the forms described in proposition \ref{simples}.
 \end{lemma}

 \begin{proof}
  The group $G_{{\varphi }}$ is polycyclic, so  the Bass-Serre tree $T$ associated to ${\Lambda }$ is a line. The set of elliptic elements of ${\Z^n}$ in ${\Lambda }$ is the kernel of a non trivial homomorphism from $\Z^n$ onto $\Z$. It is therefore a hyperplane ${H }$ in ${\Z^n}$ not containing ${x }$. Elements of $H$ act as the identity on $T$. Moreover ${\varphi }({H }) = t{H } t^{-1}$ is also an elliptic subgroup of ${\Z^n}$ in ${\Lambda }$. We obtain the inclusion ${\varphi }({H }) \subset {H }$, and so ${\varphi }$ stabilizes ${H }$.

  If ${\varphi }_{|{H }}=Id$, then the lemma holds. Let us assume ${\varphi }_{|{H }}\neq Id$ and let us show that ${\varphi }_{|{H }}$ has finite order and that ${\varphi }(x)x^{-1}$ belongs to ${H }$.

Write $G_{\varphi } = {\Z^n} \rtimes_{\varphi } \langle t\rangle$.  If the stable letter $t$ is elliptic in $\Lambda$, then it commutes with ${H }$ and so ${\varphi }_{|{H }}=Id$. The letter $t$ must therefore be hyperbolic. There exist two non-zero integers $h$ and $k$ such that ${x }^ht^k$ is elliptic. Then ${x }^ht^k$ must commute with ${H }$, but   ${x }$ also commutes with ${H }$, so $t^k$ commutes with ${H }$. One has ${\varphi }^k_{|{H }} = Id_{H }$.

  There exists an element ${g } \in {H }$ for which $t{x } t^{-1}={\varphi }({x }) = {x }^{p}{g }$ with $p = \pm 1$. But $t$ and $x$ act by translation implying that $p=+1$ and that ${\varphi }(x)x^{-1}$ belongs to ${H }$.
 \end{proof}

\begin{proof}[Proof of proposition \ref{simples}]
  By lemma \ref{1-1}, when $\varphi$ is not as in 1. and 2. the vertex group of $G_\varphi$ (seen as an HNN extension) is universally elliptic. By \cite[lemma 4.6]{GL3a}, the JSJ decomposition of $G_\varphi$ is the HNN extension $\Z^n\ast_\varphi$. 
\end{proof}

Every other group $G_\varphi$ has at least one other action:

 \begin{lemma}\label{flexible}
    Let  $({x },h_1,\dots,h_{n-1})$ be a basis of $\Z^n$ in which $\varphi$ can be written as in 1. or 2. of proposition \ref{simples}.
    Then the semi-direct product $G_{\varphi } = \mathbb Z^n \rtimes_{\varphi } \langle t\rangle $ can be decomposed as a graph of groups in which ${H }=\langle h_1,\dots, h_{n-1}\rangle$ is elliptic and ${x } h$ is hyperbolic for all $h$ in $H$.
 \end{lemma}


 \begin{proof}
The second case has already been done (see proposition \ref{TFAE} and corollary \ref{2-2}).
It remains the first case.
	
 We produce another decomposition of $G_{{\varphi }}$ as a semi-direct product. Call $k$ the order of $M$.

Define $f: G_\varphi \rightarrow \Z $   by $f(t)=1$, $f(x)=-k$, and 
$f(g)=0$ for all $g\in H$.
As $tHt^{-1}=H$ and $txt^{-1}=xh$ with $h$ in $H$, this is well defined. 

Let $G_\varphi$ act on the line by translations via   $f$.
The kernel of the action is generated by $H$ and $xt^k$ which commute. It is therefore an abelian group $\Z^n$.
 \end{proof}

\begin{proposition}
 If $\varphi$ can be written in one of the following ways 
\begin{enumerate}
    \item $\left(
          \begin{array}{cc}
            1 & 0 \\ 
            p & M
          \end{array}
          \right)$
          with $M$ an $(n-1)\times(n-1)$ matrix of finite order and $p$ in $\Z^{n-1}$,
    \item $\left(
          \begin{array}{cc}
           -1 & 0 \\ 
           p  & Id_{n-1}
          \end{array}
          \right)$ with $p$ in $\Z^{n-1}$
   \end{enumerate}
then the group $G_{\varphi}$ has trivial JSJ decomposition.
\end{proposition}

 \begin{proof}
  Let $x$, $H$ and $t$ be as in lemma \ref{flexible}. Following lemma \ref{flexible}, no element of ${x }{H } $ is universally elliptic. In every JSJ decomposition each of these elements is hyperbolic or fixes a unique vertex of the Bass-Serre tree $T$ of this decomposition.

  Assume every element of ${x } {H } $ is elliptic. Since ${E}=\langle {x }, {H } \rangle$ is abelian, they all fix the same vertex. So the group ${E}$ is elliptic and fixes a unique vertex $v\in T$. But ${E} $ is normal, so the decomposition is trivial.

  Let us now assume that there exists a hyperbolic element ${y } \in {x } {H }$. Either $t$ is elliptic and $t^2$ fixes the whole tree, or $t$ is hyperbolic and acts by a translation. In particular, there exist $p$ in $\Z$ and $q$ in $\Z^*$ such that ${y } ^pt^q$ is elliptic and belongs to an edge stabilizer. This element, hyperbolic in the initial graph of groups, is not universally elliptic, which is a contradiction.
 \end{proof}

\section{Interactions between subgraphs}\label{sectioninteraction}

For $e$ an edge of a graph of group $\Gamma$, call $\Gamma_e$ the subgraph consisting of the single edge $e$ (and its vertices) and $\Pi_e$ its fundamental group.

Until now, we have shown that the non-universally elliptic edges must be $1-1$ loops or $2-2$ edges (Corollary \ref{rigidite1}). In this section, we characterize the non-universally elliptic edges $e$. We show that necessarily $\Pi_e$ has trivial JSJ decomposition, and that adjacent edge groups are included in a specific hyperplane of $G_e$.


We first prove that two adjacent edges cannot be both non-universally elliptic. 



\begin{lemma}\label{adjacent}

Let $G=\pi_1(\Gamma)$ be a \GGBS~group and $T$ its Bass-Serre tree. Let ${v }$ be a vertex of $\Gamma$. Let ${e }$ and ${f }$ be loops of type $1-1$ or edges of type $2-2$, distinct and adjacent to ${v }$. Then  $G_{ {e }}$ and $G_{ {f }}$ are universally elliptic.
\end{lemma}

\begin{proof}
Let $\tilde {e }$, $\tilde {f }$ and $\tilde {v }$ be   preimages in $T$ of ${e }$, ${f }$ and ${v }$ respectively such that $\tilde {e }$ and $\tilde {f }$ have $\tilde {v }$ as common vertex. We call $\tilde {w_1 }$ and $\tilde {w_2 }$ the second vertices of  $\tilde {e }$ and $\tilde {f }$ respectively.
There are three cases to handle.
\begin{enumerate}
  \item The edges $\tilde {e }$ and $\tilde {f }$ are of type $1-1$.

Let $t$ and $t'$ such that $t\cdot \tilde {v }= \tilde {w_1 }$ and  $t'\cdot \tilde {v }= \tilde {w_2 }$
Take ${y }\in G_{\tilde {v }}=G_{\tilde {e }}=G_{\tilde {f }}$ and suppose ${y }$ is not universally elliptic. Let $\Gamma'$ be a graph of groups in which ${y }$ is hyperbolic. As ${y }$, $t{y } t^{-1}$ and $t'{y } t'^{-1}$ commute, the elements $t$ and $t'$ must stabilize the axis of ${y }$. So there exist $p$, $q$, $r$ and $s$ integers such that $q\neq 0$,  $s\neq 0$, and ${y }^pt^q$ and ${y }^rt'^s$ are elliptic with characteristic spaces containing the axis of ${y }$. So ${y }^pt^q$ and ${y }^rt'^s$ must commute. Yet their projections in the (topological) fundamental group of the graph $\Gamma$ generate a free group of rank $2$, which is a contradiction.

\item The edges $ \tilde {e }$ and $\tilde {f }$ are of type $2-2$.

 We show that every ${y } \in G_{\tilde {e }}$ is universally elliptic. Replacing ${y }$ by ${y }^2$ we may assume that ${y } \in G_{\tilde {f }}$. Let us fix $x_1\in G_{\tilde {w_1 }}\setminus G_{\tilde {e }}$, $x_2 \in G_{\tilde {w_2 }}\setminus G_{\tilde {f }}$ and $z\in G_{\tilde {v }}\setminus (G_{\tilde {e }}\cup G_{\tilde {f }})$. So $x_1z$ and $x_2z$ are in the centralizer of ${y }$, since $x_1$, $x_2$ and $z$ are in it. Yet those are hyperbolic elements with distinct axis crossing precisely in $\tilde {v }$, they generate a free group. Applying proposition \ref{centralisateur} ${y }$ is universally elliptic.

 \item The edge $\tilde {e }$ is of type $1-1$ and $\tilde {f }$ of type $2-2$.

Let $t$ be as in the first case. Let us fix $x\in G_{\tilde {w_2 }} \setminus G_{\tilde {f }}$ and $z \in G_{\tilde {v }} \setminus G_{\tilde {f }}$. In the same way as last case, the elements $xz$ and $txzt^{-1}$ are in the centralizer of $G_{\tilde {e }}^2$ (the set of squares of $G_{\tilde e}$) and generate a free group.
\end{enumerate}
\end{proof}

\begin{proposition}\label{cas1}

Let ${G }=\pi_1(\Gamma)$ be a \GGBS~group. Let ${e }$ be a $1-1$ loop based at a vertex ${v }$ and $\varphi\in Aut(G_v)$ the modulus of ${e }$. Let $\tilde G_{v } \subset G_{v }$ be the group generated by the groups of adjacent edges, ${e }$ excepted. Then $G_{e }$ is universally elliptic if and only if there is no decomposition $G_v=\langle x \rangle \times H$ such that 
 \begin{enumerate}
   \item \label{uns} $\tilde G_{v } \subset H$,
   \item \label{bis} $\varphi$ stabilize $H$
   \item \label{ter} \begin{itemize}[leftmargin=0.3cm, noitemsep]
          \item[-] either $\varphi(x)x^{-1}\in H$ and $\varphi$ act on $H$ with finite order,
	  \item[-] or $\varphi(x)x\in H$ and $\varphi_{|H}=Id$.
         \end{itemize}
 \end{enumerate}
\end{proposition}
 
\begin{proof}

  Assume $G_{e }$ is not universally elliptic, and let $\Gamma'$ be a decomposition in which $G_{e}$ is not elliptic. From lemma \ref{1-1}, there exist a decomposition $G_v=\langle x\rangle \times H$ satisfying \ref{bis} and \ref{ter} such that $H$ is exactly the set of all elliptic elements of $G_v$ in the decomposition $\Gamma'$. By corollary \ref{rigidite1} and lemma \ref{adjacent} , the group $\tilde G_{v }$ is universally elliptic
, so $\tilde G_{v }$ is included in $H$.

  Conversely, suppose such a decomposition  $\langle x\rangle\times H$ exists. 

Let $t$ be a stable letter of $e$.
Take $(h_1,\dots,h_{n-1})$ as basis of $H$, completing it by $x$ to make a basis $(x,h_1,\dots,h_{n-1})$ of $G_{v }$. Then $t$ act on $G_{v }$ by a linear map $
\left(
\begin{array}{cc}
\varepsilon & 0 \\ 
p & M
\end{array}
\right)$ with $\varepsilon=\pm1$ , $M$ a finite order matrix and $M=Id$ if $\varepsilon=-1$. We can apply lemma \ref{flexible}, and so there is a graph of groups $\Lambda$ of $\langle G_{v }, t\rangle$ in which $x$ is hyperbolic and $H$ is elliptic. Call $v'$ an vertex of $\Lambda$ with $H\subset G_{v'}$.
 We can construct a new graph of group decomposition $\Gamma'$ of $G$ as follows. The underlying graph is obtained by removing $e$ from $\Gamma$ and gluing $\Lambda$ by adding an edge between  $v\in\Gamma$ and $v'\in\Lambda$. We define the vertex groups in the following way:
\begin{itemize}[noitemsep]
 \item for every vertex $w$ of $\Gamma'$ coming from a vertex of $\Gamma\setminus \{v\}$, we define $G_w(\Gamma')=G_w(\Gamma)$,
\item for every edge $f$ of $\Gamma'$ coming from an edge of $\Gamma$ not adjacent to $v$, we define $G_f(\Gamma')=G_f(\Gamma)$ with the natural inclusions in the adjacent vertices,
 \item $G_v(\Gamma')=H$,
 \item for every edge $f$ of $\Gamma'$ coming from an edge of $\Gamma$ adjacent to $v$, we define $G_f(\Gamma')=G_f(\Gamma)$, with the inclusion $G_f(\Gamma') \hookrightarrow G_{v}(\Gamma')$ coming from the inclusions $G_f(\Gamma)\subset H$ coming from the assumption \ref{uns},
 \item for every vertex $w$ of $\Gamma'$ coming from a vertex of $\Lambda$ including $v'$, we define $G_w(\Gamma')=G_w(\Lambda)$,
 \item for every edge $f$ of $\Gamma'$ coming from an edge of $\Lambda$, we define $G_f(\Gamma')=G_f(\Lambda)$ with the natural inclusions in the adjacent vertices,
 \item for $f$ the edge between $\Gamma$ and $\Lambda$, we define $G_f=H$, with natural inclusions $G_f=H\xrightarrow{Id}G_{v}(\Gamma')=H$ and $G_f\hookrightarrow G_{v'}(\Gamma')$ coming from $H\subset G_{v'}(\Lambda)$. 
\end{itemize}
Using the isomorphism $G_{v'}\simeq G_{v'}\ast_H H$, we easily check that $\pi_1(\Gamma')=G$
In particular $\Gamma'$ is a decomposition of $G$ in which $x$ is hyperbolic. The element $x$ is not universally elliptic.
\end{proof}

The case of a type $2-2$ edge is similar.

\begin{proposition}\label{cas2}

Let $G = \pi_1(\Gamma)$ be a \GGBS~group. Let ${e }$ be an edge of type $2-2$ with vertices ${v }$ and ${v' }$. We identify $G_e$ with its images into $G_v$ and $G_{v'}$. Then $G_{e }$ is universally elliptic if and only if there is no hyperplane $H$ of $G_{e }$ such that $H$ is also a hyperplane of $G_{v }$ and $G_{{v' }}$, and $H$ contains all groups of adjacent edges, $G_{e }$ excepted.

If $G_e$ is not universally elliptic, a decomposition $\Gamma'$ in which $G_e$ is not elliptic may be obtain from $\Gamma$ by replacing $e$ by a $1-1$ loop whose stabilizer is not universally elliptic (see figure \ref{decomp}).
\end{proposition}

\begin{exemple}

Let $G_1$ and $G_2$ be two \GGBS~groups with defining graphs $\Gamma_1$ and $\Gamma_2$
 and $v\in \Gamma_1$, $v'\in \Gamma_2$ two vertices with groups $\Z^n$. We construct a new \GGBS~group with graph of groups $\Gamma$ as the union of $\Gamma_1$ and $\Gamma_2$ and a new edge $e$ between $v$ and $v'$. We define $G_{v}(\Gamma)=G_{v}(\Gamma_1)\times \langle a \rangle$ and $G_{v'}(\Gamma)=G_{v}(\Gamma_2)\times \langle b \rangle$. We then define the edge group $G_e$ to be $\Z^n\times \langle c\rangle$ identifying $c$ with $a^2$ in $G_{v}(\Gamma)$ and with $b^2$ in $G_{v'}(\Gamma)$ and the $Z^n$ part with $G_v(\Gamma_1)$ and $G_{v'}(\Gamma_2)$. 

The group $\pi_1(\Gamma)$ has an other decomposition $\Gamma'$, with underlying graph obtained gluing $\Gamma_1$ and $\Gamma_2$ by identifying $v$ and $v'$ together, and adding a loop $l$ over the new vertex. Define $G_v(\Gamma')=\Z^n\times\langle d \rangle$. The loop carries the HNN extension with sable letter $t$ define by $tdt^{-1}=d^{-1}$ and $tht^{-1}=h$ for all $h$ in $\Z^n$. 
To obtain the isomorphism between $\Pi_l$ and $\Pi_e$, it suffices to identify $t$ with $a$, $d$ with $ab^{-1}$, and the $\Z^n$ parts together.  

In $\Gamma$ the group $\tilde K=\Pi_e=G_{v}(\Gamma)\ast_{G_e} G_{v'}(\Gamma)$ is a extended Klein bottle group.
The $\Z^n$ part of $G_{e}(\Gamma)$ plays the role of $H$ in proposition \ref{cas2}.
 
\begin{figure}[!ht]
\begin{center}
\scalebox{1} 
{
\begin{pspicture}(0,-1.62)(10.865468,1.62)
\psbezier[linewidth=0.04,linecolor=red](8.78,-0.2)(8.58,1.2)(9.38,1.6)(9.58,1.4)(9.78,1.2)(9.78,0.6)(8.78,-0.2)
\usefont{T1}{ptm}{m}{n}
\rput(4.255,0.115){$v'$}
\usefont{T1}{ptm}{m}{n}
\rput(5.695,-0.285){$G'$}
\psline[linewidth=0.04cm](4.0,-0.2)(5.8,0.2)
\psline[linewidth=0.04cm](4.0,-0.2)(5.2,-1.6)
\psline[linewidth=0.04cm](4.0,-0.2)(5.6,-0.8)
\psline[linewidth=0.04cm,linecolor=red](1.8,-0.2)(4.0,-0.2)
\psdots[dotsize=0.12](3.94,-0.2)
\usefont{T1}{ptm}{m}{n}
\rput(8.465,-0.885){$v=v'$}
\usefont{T1}{ptm}{m}{n}
\rput(10.474999,-0.285){$G'$}
\psline[linewidth=0.04cm](8.78,-0.2)(10.58,0.2)
\psline[linewidth=0.04cm](8.78,-0.2)(9.98,-1.6)
\psline[linewidth=0.04cm](8.78,-0.2)(10.38,-0.8)
\usefont{T1}{ptm}{m}{n}
\rput(7.2749996,-0.285){$G$}
\psline[linewidth=0.04cm](8.78,-0.2)(7.58,0.8)
\psline[linewidth=0.04cm](8.78,-0.2)(6.98,-1.0)
\psdots[dotsize=0.12](8.78,-0.2)
\usefont{T1}{ptm}{m}{n}
\rput(2.055,0.115){$v$}
\usefont{T1}{ptm}{m}{n}
\rput(0.89500004,-0.085){$G$}
\psline[linewidth=0.04cm](1.8,-0.2)(0.6,0.8)
\psline[linewidth=0.04cm](1.8,-0.2)(0.0,-1.0)
\psdots[dotsize=0.12](1.8,-0.2)
\end{pspicture} 
}
\vspace{0.5cm}\\

\end{center}
\caption{The left graph represent $\Gamma$, the right one $\Gamma'$.}
\label{decomp}
\end{figure}
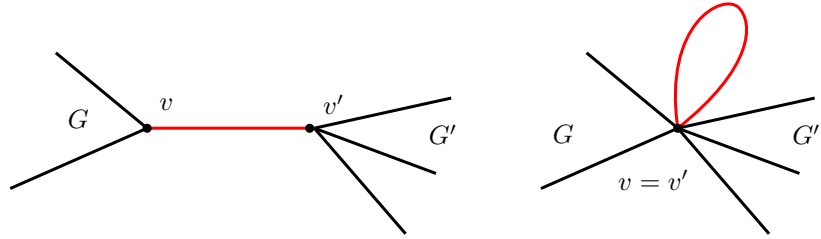

\end{exemple}
\begin{proof}
 We call $\tilde K$ the group $G_{v }\ast_{G_{e }} G_{{v' }}$. By proposition \ref{TFAE}
 $\tilde K$ is an extended Klein bottle group, so we have $\tilde K=\langle \lambda,\mu\ |\ \mathcal R\rangle \times L$ with $L$ a free abelian subgroup, $\lambda \in G_v \setminus G_e$, $\mu \in G_{v'} \setminus G_e$ and $\mathcal R$ the relators of the -twisted or not- Klein bottle group. We have $G_{e }=\langle \lambda^2,\mu^2\rangle \times L$ with $\lambda^2=\mu^2$ in the case of the untwisted Klein bottle group.

 First, assume there exists a hyperplane $H$ of $G_{e }$ which is also a hyperplane of $G_{v }$ and $G_{{v' }}$ and contains all groups of adjacent edges, $G_e$ excepted. Let us prove $G_e$ is not universally elliptic.

We can fix $x_{v }$ and $x_{{v' }}$ two elements such that $G_{v } = \langle x_{v }, H \rangle$ and $G_{{v' }}= \langle x_{{v' }}, H \rangle$. As $G_{e }$ has index 2 in both $G_{v }$ and $G_{{v' }}$, and $H$ is included in $G_{e }$, we must have $G_{e } = \langle x^2_{v }, H \rangle = \langle x^2_{{v' }}, H \rangle$. So there exists $h\in H$ such that $\tilde K=\langle x_{v },x_{{v' }}, H | x^2_{v } = x^{\pm2}_{{v' }}h, [H,x_{v }]=[H,x_{v' }]=1\rangle$. Up to taking the inverse of $x_{v'}$, one may assume $x^2_{v } = x^{2}_{{v' }}h$.

  If $h$ is a square of an element $h'$, replacing $x_{{v' }}$ by $x_{{v' }}h'$, the group $\tilde K$ may have the presentation $\tilde K=\langle x_{v }, x_{{v' }} | x_{v }^2=x_{{v' }}^2\rangle \times H$. By lemma \ref{2-2}, the group $\tilde K$ admit a decomposition in semi-direct product ({\emph i. e.} a graph of groups with one vertex and one edge) in which $H$ and $x_{v }x_{{v' }}^{-1}$ are elliptic and $x_{v }$ is hyperbolic.
 
  By collapsing $e$, we obtain a new vertex $v'$ carrying the group $\tilde K$. Since the groups of edges adjacent to $G_{v }$ et $G_{v' }$ are included in $H$, we can expand $v'$ in a $1-1$ loop with $H$ and $x_{v }x_{{v' }}^{-1}$ elliptic and $x_{v }$ hyperbolic.
  So the subgroup $G_{e }$ is not universally elliptic, and the edge ${e }$ may be replace by a loop of type $1-1$ which is not universally elliptic.

  If $h$ is not a square, up to modifying $h$ and $x_{v'}$ by a square of $H$ we can assume $h$ is primitive. Let $H'$ such that $H=\langle h \rangle\times H'$. By the Tietze transformation consisting of replacing $h$ by $x_{v'}^{-2}x_v^2$, the relations of commutation $x_vh=hx_v$ and $x_{v'}h=hx_{v'}$ become $x_vx_{{v' }}^2=x_{{v' }}^2x_v$ and $x_{v }^2x_{{v' }}=x_{{v' }}x_{v }^2$. So the group $\tilde K$ admits a presentation $\tilde K=\langle x_{v }, x_{{v' }} | x_{v }^2x_{{v' }}=x_{{v' }}x_{v }^2, x_vx_{{v' }}^2=x_{{v' }}^2x_v\rangle \times H'$. The lemma \ref{2-2} assure us to have a decomposition of $\tilde K$ in graph of groups with one vertex and one loop in which $H'$, $x_{v }^2x_{{v' }}^{-2}=h$ and $x_{v }x_{{v' }}^{-1}$ are elliptic and $x_{v }$ is hyperbolic. Then $H=\langle h, H'\rangle$ is elliptic, the same argument permits to extend the construction to the whole group.

  In both cases, the $2-2$ edge can be removed and replaced by a $1-1$ loop with the element $x_vx_{v'}^{-1}$ belonging to the edge group. This edge cannot be universally elliptic. 

  Conversely, if ${e }$ is not universally elliptic, let ${\alpha } \in G_{e }$ be a non universally elliptic element and $\Gamma'$ a graph of groups in which ${\alpha }$ is hyperbolic with axis $\mathcal A$.
 The set of the elliptics of  $G_{e }$ in $\Gamma'$ is exactly a hyperplane $H$. It remains to show that $H$ has the properties we want.

  We first claim that $H$ is exactly the set of elements of $\tilde K$ elliptic in both $\Gamma$ and $\Gamma'$.

  Let ${\beta }$ be an element of $G_{e }$, as $\lambda$ belongs to $G_v\setminus G_e$ and $\mu$ to $G_{v'}\setminus G_e$, the elements  $\lambda {\beta }$ and $\mu\lambda {\beta }\mu^{-1}$ do not commute. However they both commute with $\alpha$ so their characteristic spaces contain the axis $\mathcal A$ in the Bass-Serre tree of $\Gamma'$. Necessarily $\lambda {\beta }$ is hyperbolic: since the edge groups are abelian, if $\lambda {\beta }$ is elliptic in $\Gamma'$, the elements $\lambda {\beta }$ and $\mu\lambda {\beta }\mu^{-1}$ should commute.  And so each element of $\lambda G_{e }$ is hyperbolic. The same argument works to show that every element in $\mu G_{e }$ is hyperbolic in $\Gamma'$.

  Yet every elliptic element of $\tilde K$ for the decomposition $\Gamma$ is conjugated to an element of $G_{v }$ or $G_{{v' }}$, and so is conjugate to an element of $$G_{e } \cup \lambda G_{e } \cup \mu G_{e }.$$
  The elements of $\tilde K$  elliptic in both $\Gamma$ and $\Gamma'$ belongs to $G_e$, hence to $H$. This prove the claim.

  By the claim $H$ contains every elliptic element of $G_{v }$ and $G_{{v' }}$ for the decomposition $\Gamma'$. As $H$ is a subgroup of rank $n-1$ in both $G_{v }$ and $G_{{v' }}$, it is a hyperplane in both $G_{v }$ and $G_{{v' }}$. 
  Moreover, by corollary \ref{rigidite1} and lemma \ref{adjacent}, the group of every edge adjacent to ${e }$ is universally elliptic. So each of these edge groups must be include in $H$.
\end{proof}

\section{\texorpdfstring{JSJ decomposition of \GGBS~group}{}}\label{finalsection}

Let $G=\pi_1(\Gamma)$ be a \GGBS~group and ${v }$ a vertex of $\Gamma$. Let $\tilde G_{v }\subset G_v$ the group generated by the groups of edges adjacent to ${v }$. If $G_{v } / \tilde G_{v }$ is infinite, then there exists a hyperplane $H$ of $G_{v }$ containing $\tilde G_{v }$. 
Then one can replace the vertex ${v }$ by an HNN extension with vertex and edge groups equal to $H$ and the modulus of the edge equal to the identity. 
Call ${v' }$ and ${e }$ the new vertex and the new edge. We have the equality $G_{{v' }}=G_{{e }}=H$, and so ${e }$ is of type $1-1$. We call this construction the \emph{expansion} of $G_{v }$ over $H$ (see figure \ref{expansion}).

\begin{figure}[!ht]
\begin{center}
\begin{pspicture}(1,-0.8)(11.507187,1.8)
\psline[linewidth=0.04cm,dotsize=0.07055555cm 2.0]{*-*}(1.761875,0.0)(4.161875,0.0)
\psline[linewidth=0.04cm,dotsize=0.07055555cm 2.0]{*-*}(7.561875,0.0)(9.761875,0.0)
\psellipse[linewidth=0.04,dimen=outer](10.561875,0.0)(0.8,0.8)
\usefont{T1}{ptm}{m}{n}
\rput(1.8354688,0.515){$\langle a, b \rangle$}
\usefont{T1}{ptm}{m}{n}
\rput(4.2254686,0.515){$\langle c, d \rangle$}
\rput(4.2254686,-0.285){$v$}
\usefont{T1}{ptm}{m}{n}
\rput(7.4354687,0.515){$\langle a, b \rangle$}
\usefont{T1}{ptm}{m}{n}
\rput(9.395469,0.515){$\langle c \rangle$}
\rput(9.595469,-0.285){$v'$}
\usefont{T1}{ptm}{m}{n}
\rput(11.695469,-0.085){$\langle c \rangle$}
\rput(11.2,0.715){$e$}
\usefont{T1}{ptm}{m}{n}
\rput(3.1454687,-0.485){$\langle b\rangle =\langle c\rangle$}
\usefont{T1}{ptm}{m}{n}
\rput(8.545468,-0.285){$\langle b\rangle =\langle c\rangle$}
\usefont{T1}{ptm}{m}{n}
\end{pspicture} \vspace{0.5cm}\\
\end{center}
\caption{Expansion of the right vertex over $\langle c\rangle$.}
\label{expansion}
\end{figure}
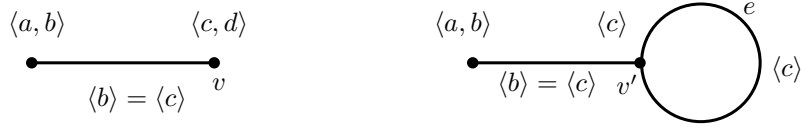

\begin{lemma}\label{stable}

  Let $G=\pi_1(\Gamma)$ be a \GGBS~group. Let ${v }$ be vertex,
 if $G_{v }/{\tilde G_{v }}$ is virtually cyclic, then there is a unique expansion of $G_{v }$ and it creates a universally elliptic edge.
 \end{lemma}

 \begin{proof}

  Let $H$ be the unique hyperplane of $G_{v }$ containing $\tilde G_{v }$.
 Let $\Gamma'$ be the graph of groups obtained by expansion of $G_{v }$ over $H$.  From lemma \ref{adjacent}, the adjacent edges are universally elliptic. Moreover, the group $\tilde G_v$ generated by all those groups has finite index in $G_{{v' }}=H=G_e$, 
 so the edge ${e }$ is universally elliptic.
 \end{proof}

 \begin{lemma}\label{ecrase}

  Let $G=\pi_1(\Gamma)$ be a  \GGBS~group. Let ${e }$ be a non-universally elliptic edge, and ${v }$ the vertex obtained by collapsing ${e }$. Let $\tilde G_{v }\subset G_v$ be the group generated by groups of edges adjacent to ${v }$.
 Denote $\langle\langle \tilde G_{v } \rangle\rangle$ the normal closure of $\tilde G_{v }$ in $G_{v }$. Then $G_{v }/\langle\langle \tilde G_{v } \rangle\rangle$ is not virtually cyclic.
 \end{lemma}

 \begin{proof}

  By proposition \ref{cas2}, we may assume that ${e }$ is a loop of type $1-1$. The group $G_{v }$ is a semi-direct product $\Z^n\rtimes_\varphi \Z$. By proposition \ref{cas1}, the subgroup $\tilde G_{v }$ is included in a hyperplane $H$ of $\Z^n$ stabilized under the action of $\varphi$. In particular $H$ is normal in $G_v$ and $\langle\langle \tilde G_{v } \rangle\rangle \subset H$. So there exists a projection from $G_{v }/\langle\langle \tilde G_{v } \rangle\rangle$ onto $G_{v }/H \simeq \Z\rtimes \Z$ which is virtually $\Z^2$. The group $G_{v }/\langle\langle \tilde G_{v } \rangle\rangle$ is not virtually cyclic.
 \end{proof}

 \begin{theorem}\label{final}

    Let $G=\pi_1(\Gamma)$ a \GGBS~group. For ${v }$ a vertex, let $\tilde G_{v }$ be the group generated by groups of edges adjacent to ${v }$. A JSJ decomposition of $G$ can be obtained from $\Gamma$ by collapsing the edges carrying a non-universally elliptic group and expanding the groups $G_{v }$ such that $G_{v } / \tilde G_{v }$ is virtually cyclic.
 \end{theorem}

By lemmas \ref{stable} and \ref{ecrase}, the two operations of collapsing and expanding do not interact, so the construction is well define.
 \begin{proof}

  Let $\Lambda$ be the obtained graph of groups. According to lemmas \ref{stable} and \ref{ecrase}, this graph does not contain neither vertex ${v }$ such that $G_{v } / \tilde G_{v }$ is virtually cyclic nor edges carrying a non-universally elliptic group. In particular $\Lambda$ is universally elliptic.
 
It remains to prove the maximality of $\Lambda$.

  Let $\Lambda'$ be another decomposition in which all edge groups are universally elliptic. We have to show that every elliptic element in $\Lambda$ is elliptic in $\Lambda'$. Then by \cite[lemma 3.6]{GL3a} the decomposition $\Lambda$ is a JSJ decomposition of $G$. 

Suppose there exists an element $x$, elliptic in $\Lambda$ and  hyperbolic in $\Lambda'$. We will show there exists an edge group of $\Lambda'$ containing a non-universally elliptic element, a contradiction. 

There are two cases to consider:
   \begin{enumerate}
     \item If $x$ belongs to a vertex group of $\Lambda$ carrying an abelian group.

      Let $v$  be a vertex stabilized by $x$ in the Bass-Serre tree of $\Lambda$.
      Then no power of $x$ belongs to $\tilde G_{v }$ which is universally elliptic. As $G_{v }$ is abelian but not elliptic in $\Lambda'$, it acts non trivially on the axis of $x$ via an map $\varphi:~G_{v } \rightarrow \Z$ with $\tilde G_{v } \subset \ker \varphi$. 

    We shall find a decomposition $G_v=\langle y \rangle\times H$ such that $\tilde G_{v } \subset H$ and $y\in\ker \varphi$. Then expanding $G_{v }$ over $H$ in $\Lambda$, the element $y$ will be hyperbolic, contradicting the universal ellipticity of $\Lambda'$.

    Let $\bar G_{v }$ be the set of elements of $\ker \varphi$ with a power in $\tilde G_{v }$. As $G_{v } / \tilde G_{v }$ is infinite and not virtually cyclic, the set $\ker \varphi \setminus \bar G_{v }$ is non-empty. Take $y$ a primitive element in $\ker \varphi \setminus \bar G_{v }$.
    Taking any hyperplane $H$ of $G_v$ containing $\tilde G_v$ such that $G_v=\langle y \rangle\times H$, we obtain the decomposition.

     \item If $x$ belongs to a vertex group of $\Lambda$ carrying a non-abelian group.

  From proposition \ref{cas2}, we may assume that all non-universally elliptic edges of $\Gamma$ are loops with type $1-1$. 
Then $x$ belongs to a group $\Pi_l$ for $l$ a non-universally elliptic loop of $\Gamma$. Call $w$ the vertex and $t$ a stable letter of $l$. Let $h$ be an element of $G_{w }$ non universally elliptic. As $\Pi_l$ is polycyclic, it acts on the axis of $x$ in $\Lambda'$. If $h$ is elliptic in $\Lambda'$, then $h^2$ fixes the axis and so $\Lambda'$ is not universally elliptic. If $h$ is hyperbolic , there exists two integers $p$ and $q\neq 0$ such that $h'=h^pt^q$ is elliptic in $\Lambda'$ (the integer $p$ may be equal to $0$ if $t$ is already elliptic in $\Lambda'$). Up to take the square of $h'$, we may assume it fixes the axis pointwise. As $q$ is different from 0, the element $h'$ is hyperbolic in $\Gamma$. There exists an edge of $\Lambda'$ which is not universally elliptic.
   \end{enumerate}

\end{proof}

\begin{theorem}
 Let $G=\pi_1(\Gamma)$ a \GGBS~group, and $n\in \N$. Suppose that $\Gamma$ is a JSJ decomposition over free abelian groups. Then a JSJ decomposition over free abelian groups of rank $\leq n$ may be obtained by collapsing every edge of $\Gamma$ with group of rank $> n$.
\end{theorem}

We will use the notation $n$-JSJ decomposition and $n$-universally elliptic for decomposition over group of rank $\leq n$ and $\infty$-JSJ decompositions and $\infty$-universally elliptic for decomposition over all free abelian groups.

\begin{proof}
 By \cite[Proposition 7.1]{GL3a}, we know that an $n$-JSJ decomposition may be obtained by refining an $\infty$-JSJ decomposition, and then collapsing all edges carrying a group of rank $> n$. So it suffices to show that if a vertex of $\Gamma$ 
 can be decomposed over a group of rank $\leq n$, then this decomposition is not $n$-universally elliptic.

Let $v$ be a vertex of $\Gamma$. 
Its vertex group is a semi-direct product $G_v=\Z^p\rtimes \Z$ (the product may be direct). 
The only non-trivial reduced decompositions of $G_v$ are as $1-1$ loop or a $2-2$ edge with edge group $\Z^p$.

Call $\Lambda$ a decomposition of $G$ refining $\Gamma$ and an $n$-JSJ decomposition $\Delta$, and $\Lambda_v$ the subgraph of $\Lambda$ whose projection on $\Gamma$ is $v$. Assume $\Lambda_v$ has a minimal number of edges over decomposition refining $\Gamma$ and $\Delta$. If $\Lambda_v$ contains a non-reduced edge then after collapsing all edges with groups of rank $>n$ in $\Lambda$ the edge stays unreduced (if an edge having same group as one of its endpoint $v$ is not collapsed then $rank(G_v)\leq n$ and no edge adjacent to $v$ is collapsed). As $\Lambda_v$ has a minimal number of edges, it is reduced. 


If $\Lambda_v$ is not a trivial decomposition of $G_v$, as $G_v$ cannot be decomposed over groups of rank $>p$, if $\Lambda$ is $n$-universally elliptic then $\Lambda$ is $\infty$-universally elliptic.
This is a contradiction with the fact that $\Gamma$ is a $\infty$-JSJ decomposition.
\end{proof}

\section{Contruction of the JSJ decomposition}\label{algo}
We now describe an algorithm which gives the abelian JSJ decomposition of a \GGBS~group.

We assume that a \GGBS~group is given under the form of a \GGBS~decomposition, that is, given by the description of the vertex groups, edge groups, and the inclusion maps of the edge groups into the vertex groups. Given an edge $e$ and an endpoint $v$ of $e$, we may decide if $G_e$ is of finite index in $G_v$ and if so, compute this index. In particular, we may detect $1-1$ loops and $2-2$ edges.


\begin{theorem}\label{thalgo}
 The construction of a JSJ decomposition of a \GGBS~group is algorithmic.
\end{theorem}

Let $\Gamma$ be a \GGBS~decomposition and $l$ be a $1-1$ loop with base point $v$ and modulus $\varphi$ of finite order which is not conjugate to $\left(
          \begin{array}{cc}
            -1 & 0 \\ 
            p & Id_{n-1}
          \end{array}
          \right)$ with $p\in\Z^{n-1}$. Call $n$ the rank of $G_v$. Call $\tilde G_v$ the subgroup of $G_v$ generated by groups of edges adjacent to $v$ excepting $l$. Call $E(\varphi)\subset G_v\otimes \mathbb C$ the subspace generated by all eigenspaces of $\varphi$ associated to eigenvalues different from $1$. As $\varphi$ is of finite order, thus diagonalizable, the dimension of $E(\varphi)$ is equal to $n$ minus the dimension of $E_1(\varphi)$, the eigenspace associated to $1$. Define $\bar G_v$ as  the smallest  $\varphi$-invariant subgroup of $G_v$ containing $\tilde G_v$ and $E(\varphi)\cap G_v$.

\begin{lemma}\label{casparticulier}
With the previous notations, the loop $l$ is universally elliptic if and only if $\bar G_v$ is of rank $n$.
\end{lemma}

\begin{proof}
          First assume $l$ is not universally elliptic.
          By proposition \ref{cas1}, as $\varphi$ is not conjugate to $\left(
          \begin{array}{cc}
            -1 & 0 \\ 
            p & Id
          \end{array}
          \right)$, in a well-chosen basis $(x,h_1\dots,h_{n-1})$ the modulus $\varphi$ is  of the form $\left(
          \begin{array}{cc}
            1 & 0 \\ 
            p & M
          \end{array}
          \right)$ with $M$ of finite order and $\tilde G_v$ contained in $H=\langle h_1\dots,h_{n-1}\rangle$. Moreover,  as the action of $\varphi$ on $G_v/H$ is trivial, every eigenvector in $G_v\otimes \mathbb C$ with eigenvalue $\neq 1$ must belong to $H\otimes \mathbb C$.
Thus the rank of $\bar G_v$ must be at most $n-1$.

Conversely,  if the rank of $\bar G_v$ is $<n$, then we may construct a group $L$ of rank $n-1$, containing $\bar G_v$ and stable under the action of $\varphi$, by adding to $\bar G_v$ eigenvectors associated to $1$. This is possible since $\dim(E(\varphi)\oplus E_1(\varphi))=n$. Call $H$ the hyperplane containing $L$, and $x\in G_v$ such that $G_v=\langle x\rangle\times H$. Then the matrix of $\varphi$ in a basis $(x, h_1,\dots, h_{n-1})$ where $H=\langle h_1,\dots, h_{n-1}\rangle$, is of the form $\left(
          \begin{array}{cc}
            1 & 0 \\ 
            p & M
          \end{array}\right)$.
          By proposition \ref{cas1}, the loop $l$ is not universally elliptic.
\end{proof}

We may now describe the algorithm.

\begin{proof}[Proof of theorem \ref{thalgo}]
 By theorem \ref{theoremintro}, there are three different algorithms to construct. The first blows up certain vertices into a loop, and the other two decide when a $1-1$ loop and a $2-2$ edge is universally elliptic. Recall that we must collapse non universally elliptic edges.

\begin{itemize}
\item The first algorithm is very simple. Given a vertex $v$, it computes the group $\tilde G_v$ generated by adjacent edge groups. If $rank(\tilde G_v)\neq rank(G_v)-1$, then it leaves $v$ unchanged.
 If the rank of $\tilde G_v$ is equal to the rank of $G_v$ minus one, it finds a primitive element $a$ such that no power is in $\tilde G_v$ (the algorithm of the Smith normal form works for example), and changes $v$ to a new vertex $v'$ and a loop $l$ with $G_l=G_{v'}$ equal to the set of elements of $G_v$ with a power in $\tilde G_v$, and $a$ the stable letter.

\item The second algorithm has to decide whether a $2-2$ edge $e$ is universally elliptic or not. 

Call $v$ and $v'$ the two endpoints of $e$. By proposition \ref{cas2}, if the adjacent edge groups are not all contained in $G_e$,  then $e$ is universally elliptic. Otherwise, call $\tilde G_e$ the subgroup of $G_e$ generated by all these groups, and $\bar G_e$ the set of elements of $G_e$ with a power in $\tilde G_e$. Then $e$ is universally elliptic if and only if $G_v/\bar G_e$ or $G_{v'}/\bar G_e$ has torsion. This is decidable by looking at the Smith normal form of $\bar G_e$ in $G_v$ and $G_{v'}$.
 
\item The third algorithm decides whether a $1-1$ loop $l$ is universally elliptic or not.

  Call $v$ the base point of $l$, call $\varphi$ the modulus of $l$ and $n$ the rank of $G_v$. One first computes the group $\tilde G_v$ generated by adjacent edge groups, $G_l$ excepted.
If $l$ is not universally elliptic, necessarily by proposition \ref{cas1} the modulus $\varphi$ acts on a hyperplane of $G_v$ with finite order. As we have a minoration for Euler's phi function $\phi$ given by $\phi(p)\geq \sqrt{\frac{p}{2}}$, the modulus $\varphi$ acts on a hyperplane with finite order if and only if $\varphi^{(2m^2)!}$ acts trivially on a hyperplane. We therefore consider the eigenspace $E$ of $\varphi^{(2m^2)!}$ associated to the eigenvalue $1$.

There are three cases.
\begin{itemize}
 \item either $E$ has rank less than $n-1$,
 \item or $E$ has rank $n-1$.
 \item or $E=G_v$. 
\end{itemize}
In the first case, the loop $l$ is universally elliptic.

In the second case $E$ is a $\varphi$-invariant hyperplane. As the only hyperplane on which $\varphi$ acts with finite order is $E$, it suffices to check if $\tilde G_v$ is included in $E$, and if $\varphi$ acts trivially on $G_v/E$.
By proposition \ref{cas1}, the loop $l$ is not universally elliptic if and only if those two properties hold.

In the third case, we have to decide if there exists a hyperplane $H$ containing $\tilde G_v$ with properties as in proposition \ref{cas1}. 

We first decide if $\varphi$ is conjugate to $\left(
          \begin{array}{cc}
            -1 & 0 \\ 
            p & Id
          \end{array}
          \right)$ for some $p$ in $\Z^{n-1}$. This is decidable because  matrices conjugate to a matrix of the previous form are exactly the ones with determinant $-1$ and an eigenspace associated to eigenvalue $1$ of dimension $n-1$.
          
If $\varphi$ is conjugate to $\left(
          \begin{array}{cc}
            -1 & 0 \\ 
            p & Id
          \end{array}
          \right)$, then $l$ is not universally elliptic if and only if $\varphi$ act trivially on $\tilde G_v$.

If $\varphi$ is not conjugate to $\left(
          \begin{array}{cc}
            -1 & 0 \\ 
            p & Id
          \end{array}
          \right)$,  call $P(X)$ the polynomial $\frac{X^{(2m^2)!}-1}{X-1}$, and compute the subgroup $F=\langle \ker(P(\varphi)),\tilde G_v\rangle$.
As $rank(\varphi^k(F))$ is first strictly increasing, and then stationary, the group $M=\varphi^n(F)$ is the smallest $\varphi$-invariant subgroup containing $F$.
We compute $m$ the rank of $M$.
By lemma \ref{casparticulier}, the edge $l$ is universally elliptic if and only if $m=n$.
 \end{itemize}
\end{proof}

\bibliographystyle{plain}
\bibliography{abelianjsj}

\begin{thebibliography}{1}

\bibitem{bow}
B.~H. B{\sc{owditch}}.
\newblock Cut points and canonical splittings of hyperbolic groups.
\newblock {\em Acta Math.}, 180(2):145--186, 1998.

\bibitem{DunSa2}
M.~J. D{\sc{unwoody}} and M.~E. S{\sc{ageev}}.
\newblock J{SJ}-splittings for finitely presented groups over slender groups.
\newblock {\em Invent. Math.}, 135(1):25--44, 1999.

\bibitem{for}
M.~{\sc Forester}.
\newblock Splittings of generalized {B}aumslag-{S}olitar groups.
\newblock {\em Geometriae Dedicata}, 121:43--59, August 2006.

\bibitem{FuPa}
K.~F{\sc{ujiwara}} and P.~P{\sc{apasoglu}}.
\newblock J{SJ}-decompositions of finitely presented groups and complexes of
  groups.
\newblock {\em Geom. Funct. Anal.}, 16(1):70--125, 2006.

\bibitem{GL3a}
V.~{\sc Guirardel} and G.~{\sc Levitt}.
\newblock {JSJ} decompositions: definitions, existence, uniqueness. {I}: The
  {JSJ} deformation space.
\newblock {\em http://arxiv.org/abs/0911.3173}, 2009.

\bibitem{Krop}
P.~H. K{\sc{ropholler}}.
\newblock An analogue of the torus decomposition theorem for certain poincar\'e
  duality groups.
\newblock {\em Proc. London Math. Soc.}, 60(3):503--529, 1990.

\bibitem{RiSe}
E.~R{\sc{ips}} and Z.~S{\sc{ela}}.
\newblock Cyclic splittings of finitely presented groups and the canonical
  {JSJ} decomposition.
\newblock {\em Ann. of Math. (2)}, 146(1):53--109, 1997.

\bibitem{Sela}
Z.~S{\sc{ela}}.
\newblock Structure and rigidity in ({G}romov) hyperbolic groups and discrete
  groups in rank {$1$} {L}ie groups. {II}.
\newblock {\em Geom. Funct. Anal.}, 7(3):561--593, 1997.

\end{thebibliography}

\vspace{0.3cm}
\begin{flushleft}
\noindent \textsc{Benjamin Beeker}\\
LMNO,\\
Universit\'e de Caen BP 5186\\
F 14032 Caen Cedex \\
France\\
{\tt benjamin.beeker@math.unicaen.fr}
\end{flushleft}
\end{document}